\documentclass{article}
\usepackage{amssymb,amsmath}
\usepackage{graphicx,enumerate}
\usepackage[utf8]{inputenc}
\title{ Reticulation of  Quasi-commutative Algebras }

\author{George Georgescu \\ \footnotesize University of Bucharest\\ \footnotesize Faculty of Mathematics and Computer Science\\ \footnotesize Bucharest, Romania\\ \footnotesize Email: georgescu.capreni@yahoo.com}

\date{}

\begin{document}
\maketitle

\begin{abstract}

The commutator operation in a congruence-modular variety $\mathcal{V}$ allows us to define the prime congruences of any algebra $A\in \mathcal{V}$  and the prime spectrum $Spec(A)$ of $A$. The first systematic study of this spectrum can be found in a paper by Agliano, published in Universal Algebra (1993).

The reticulation of an algebra $A\in \mathcal{V}$  is a bounded distributive algebra $L(A)$, whose prime spectrum (endowed with the Stone topology) is homeomorphic to $Spec(A)$ (endowed with the topology defined by Agliano). In a recent paper, C. Mure\c{s}an and the author defined the reticulation for the algebras $A$ in a semidegenerate congruence-modular variety $\mathcal{V}$, satisfying the hypothesis $(H)$: the set $K(A)$ of compact congruences of $A$ is closed under commutators. This theory does not cover the Belluce reticulation for non-commutative rings. In this paper we shall introduce the quasi-commutative algebras in a semidegenerate congruence-modular variety $\mathcal{V}$ as a generalization of the Belluce quasi-commutative rings. We define and study a notion of reticulation for the quasi-commutative algebras such that the Belluce reticulation for the quasi-commutative rings can be obtained as a particular case. We prove a characterization theorem for the quasi-commutative algebras and some transfer properties by means of the reticulation.

\end{abstract}


\newtheorem{definitie}{Definition}[section]
\newtheorem{propozitie}[definitie]{Proposition}
\newtheorem{remarca}[definitie]{Remark}
\newtheorem{exemplu}[definitie]{Example}
\newtheorem{intrebare}[definitie]{Open question}
\newtheorem{lema}[definitie]{Lemma}
\newtheorem{teorema}[definitie]{Theorem}
\newtheorem{corolar}[definitie]{Corollary}

\newenvironment{proof}{\noindent\textbf{Proof.}}{\hfill\rule{2mm}{2mm}\vspace*{5mm}}

\section{Introduction}

 \hspace{0.5cm} The reticulation of a commutative ring $R$ is a remarkable construction in commutative algebra. With each commutative ring $R$ is associated a bounded distributive lattice $L(A)$, fulfilling the following property; the prime spectrum of $L(A)$ (with the Stone topology) and the prime spectrum of $R$ (with the Zariski topology) are homeomorphic spaces (see \cite {Simmons}, \cite{Johnstone}). By using this property a lot of notions and results are exported from bounded distributive lattices to commutative rings and vice-versa (see \cite{Al-Ezeh2}, \cite{Dickmann}, \cite{Johnstone}, \cite{Simmons}). An axiomatic definition of an arbitrary (unital) ring R was proposed by Belluce in \cite{Belluce}. He introduced a new class of rings (named quasi-commutative rings) and proved that the ring $R$ admits a reticulation if and only if it is quasi-commutative (see Theorem 4 of \cite{Simmons}). The theory was developed in \cite{Belluce1} and \cite{Klep}, where, by using the reticulation, some algebraic and topological results on the spectra of rings were obtained.

 Inspired by ring theory, some notions of reticulation were defined for other concrete algebraic structures: MV-algebras \cite{Belluce0}, BL- algebras \cite{Leustean}, residuated lattices \cite{Muresan}, etc. To define a notion of reticulation for universal algebras is a natural problem. The commutator theory, developed by Fresee and McKenzie \cite{g} in a congruence-modular variety $\mathcal{V}$, allows us to define the prime congruences and the prime spectrum $Spec(A)$ associated with any object of $\mathcal{V}$ (see \cite{Agliano}). We shall denote by $Spec_Z(A)$ the prime spectrum $Spec(A)$ endowed with the topology defined by Agliano in \cite{Agliano} (of course, this topology is a Zariski-style topology).

 The paper \cite{GM2} was an attempt to define a reticulation for algebras in a semidegenerate congruence-modular variety $\mathcal{V}$ \cite{Kollar}. In order to build the reticulation $L(A)$ of an algebra $A\in\mathcal{V}$ we assumed the following hypothesis:

 $\bullet$ $(H)$: The set $K(A)$ of compact congruences of $A$ is closed under the commutator operation.

 In the presence of $(H)$, we proved in \cite{GM2} that the prime spectrum $Spec_Z(A)$ of the algebra $A$ and the prime spectrum $Spec_{Id,Z}(L(A))$ of the bounded distributive lattice $L(A)$ are homeomorphic. This fact ensures a lot of transfer properties from algebras to bounded distributive lattices and vice-versa (see \cite{GM2}, \cite{GKM}, etc). The hypothesis $(H)$ is verified by the commutative rings, but it is not valid in arbitrary rings. Then the theory developed in \cite{GM2} does not cover the Belluce reticulation of non-commutative rings. Moreover, the neo-commutative rings, introduced by Kaplansky in \cite{Kaplansky}, fulfill $(H)$. Then the construction of reticulation from \cite{GM2} and the related properties can be applied to the neo-commutative rings.

 In this paper we study a notion of reticulation in the framework of a semidegenerate congruence-modular variety $\mathcal{V}$. In order to build a bounded distributive lattice $L(A)$ for any $A\in\mathcal{V}$, we generalize an idea of \cite{Belluce}: we start the construction of $L(A)$ with the set $C(A)$ generated by $K(A)$ under the commutator operations and the finite joins. $L(A)$ will be the quotient $C(A)/\equiv$, where $\equiv$ is an equivalence relation defined by the radical operation.

 In general, the prime spectra $Spec_Z(A)$ and $Spec_{Id,Z}(L(A))$ are not homeomorphic. Then the main problem is to find classes of algebras in $\mathcal{V}$ for which these prime spectra are homeomorphic.

 In this paper we shall give an answer to this question. We shall define the quasi-commutative algebras in a semidegenerate congruence-modular variety $\mathcal{V}$ as generalization of the quasi-commutative rings (introduced by Belluce in \cite{Belluce}) and we shall prove that these algebras offer a suitable setting for developing a good reticulation theory.

 Now we shall present the content of paper. Section 2 contains some preliminary matter: the commutators in congruence-modular varieties, radicals of congruences and the prime spectrum of a congruence-modular algebra (see \cite{Fresee}, \cite{Agliano}).

 Section 2 concerns the reticulation theory in a semidegenerate congruence-modular variety $\mathcal{V}$. For each $A\in \mathcal{V}$ we define the bounded distributive lattice $L(A)$ and we investigate the relationship between the congruences of $A$ and the ideals of $L(A)$. We prove the existence of reticulation for the quasi-commutative algebras in $\mathcal{V}$. A characterization of the quasi-commutative algebras in algebraic and topological terms is obtained (see Theorem 3.26).

A first result of Section 3 establishes a Boolean isomorphism between the Boolean algebra $B(Con(A))$ of complemented congruences of a semiprime algebra $A\in \mathcal{V}$ and the Boolean center $B(L(A))$ of the reticulation $L(A)$ (see Proposition 4.4). We prove some preservation results for annihilators in semiprime algebras (Proposition 4.7 and 4.8), then we obtain a characterization theorem for the minimal prime congruences of any algebra $A\in \mathcal{V}$ (Theorem 4.14).

\section{Preliminaries}

 \hspace{0.5cm} In this section we shall recall some basic notions and results in a congruence-modular algebras (commutator operation \cite{Fresee}, radical of a congruence, prime congruences and topology of the prime spectrum \cite{Agliano}, \cite{GM2}). The basic references for the matter of this section are the monographs \cite{Birkhoff}, \cite{Burris}, \cite{Fresee}.

 Let $\tau$ be a finite signature of universal algebras. Throughout this paper we shall assume that the algebras have the signature $\tau$. Let $A$ be an algebra and $Con(A)$ the complete lattice of its congruences; $\Delta_A$ and $\nabla_A$ shall be the first and the last elements of $Con(A)$. If $X\subseteq A^2$ then $Cg_A(X)$ will be the congruence of $A$ generated by $X$; if $X = \{(a,b)\}$ with $a, b\in A$ then $Cg_A(a,b)$ will denote the (principal) congruence generated by $\{(a,b)\}$. We shall denote by $PCon(A)$ the set of principal congruences of $A$. $Con(A)$ is an algebraic lattice: the finitely generated congruences of $A$ are its compact elements. $K(A)$ will denote the set of compact congruences of $A$. We observe that $K(A)$ is closed under finite joins of $Con(A)$ and $\Delta_A\in K(A)$.

 For any $\theta\in Con(A)$, $A/\theta$ is the quotient algebra of $A$ w.r.t. $\theta$; if $a\in A$ then $a/\theta$ is the congruence class of $a$ modulo $\theta$. We shall denote by $p_\theta:A\rightarrow A/\theta$ the canonical surjective $\tau$ - morphism $p_\theta(a) = a/\theta$, for all $a\in A$.

 Let  $\mathcal{V}$  be a congruence - modular variety of $\tau$ - algebras. Following \cite{Fresee}, p.31, the commutator is the greatest operation $[\cdot,\cdot]_A$ on the congruence lattices $Con(A)$ of members $A$ of $\mathcal{V}$ such that for any surjective morphism $f:A\rightarrow B$ of $\mathcal{V}$ and for any $\alpha,\beta\in Con(A)$, the following conditions hold:

 (2.1) $[\alpha,\beta]_A\subseteq \alpha\bigcap \beta$;

 (2.2) $[\alpha,\beta]_A\lor Ker(f)$ = $f^{-1}([f(\alpha\lor Ker(f)),f(\beta\lor Ker(f))]_B)$.

 By \cite{Fresee} we know that the commutator operation is commutative, increasing in each argument and distributive with respect to arbitrary joins. If there is no danger of confusion then we write $[\alpha,\beta]$ instead of $[\alpha,\beta]_A$.

 \begin{propozitie}\cite{Fresee}
 For any congruence - modular variety $\mathcal{V}$ the following are equivalent:
\newcounter{nr}
\begin{list}{(\arabic{nr})}{\usecounter{nr}}
\item  $\mathcal{V}$ has Horn - Fraser property: if $A, B$ are members of $\mathcal{V}$ then the lattices $Con(A\times B)$ and $Con(A)\times Con(B)$ are isomorphic;
\item $[\nabla_A,\nabla_A] = \nabla_A$, for all $A\in \mathcal{V}$;
\item $[\theta,\nabla_A] = \theta$, for all $A\in \mathcal{V}$ and $\theta\in Con(A)$.
\end{list}
\end{propozitie}

Following \cite{Kollar}, a variety $\mathcal{V}$ is semidegenerate if no nontrivial algebra in $\mathcal{V}$ has one - element subalgebras. By \cite{Kollar}, a variety $\mathcal{V}$ is semidegenerate if and only if for any algebra $A$ in $\mathcal{V}$, the congruence $\nabla_A$ is compact.

\begin{propozitie}\cite{Agliano}
If $\mathcal{V}$ is a semidegenerate congruence - modular variety then for each algebra $A$ in $\mathcal{V}$ we have $[\nabla_A,\nabla_A] = \nabla_A$.
\end{propozitie}

Let $A$ be a semidegenerate congruence - modular algebra. One can define on the lattice $Con(A)$ a residuation operation ( = implication) $\alpha \rightarrow \beta = \bigvee \{ \gamma|[\alpha,\gamma] \subseteq \beta\}$ and an annihilator operation ( = negation) $\alpha^{\bot } = \alpha^{\bot_A } =  \alpha \rightarrow \Delta_A =\bigvee \{\gamma |[\alpha,\gamma] = \Delta_A\}$. The implication $\rightarrow$ fulfills the usual residuation property: for all $\alpha,\beta,\gamma\in Con(A)$, $\alpha\subseteq \beta\rightarrow\gamma$ if and only if $[\alpha,\beta]\subseteq\gamma$. We remark that $(Con(A), \lor, \land, [\cdot,\cdot], \rightarrow, \Delta_A, \nabla_A)$ is a commutative and integral complete $l$ - groupoid (see \cite{Birkhoff}).

Now let us fix an algebra $A$ in a semidegenerate congruence - modular variety $\mathcal{V}$.

\begin{lema}\cite{GM2}
 For all congruences $\alpha, \beta, \gamma, $ the following hold:
\usecounter{nr}
\begin{list}{(\arabic{nr})}{\usecounter{nr}}
\item $\alpha\lor \beta = \nabla_A$ implies $[\alpha,\beta] = \alpha\bigcap \beta$;
\item $\alpha\lor \beta = \alpha\lor \gamma = \nabla_A$ implies $\alpha\lor [\beta,\gamma] = \alpha\lor (\beta\bigcap \gamma) = \nabla_A$;
\item $\alpha\lor \beta = \nabla_A$ implies $[\alpha,\alpha]^n\lor [\beta,\beta]^n = \nabla_A$, for all integer $n > 0$;

\end{list}
\end{lema}

For all congruences $\alpha, \beta \in Con(A)$ and for any integer $n\geq 1$ we define by induction the congruence $[\alpha,\beta]^n$: $[\alpha,\beta]^1$ = $[\alpha,\beta]$ and $[\alpha,\beta]^{n+1} =[[\alpha,\beta]^n,[\alpha,\beta]^n]$.

Following \cite{Fresee}, p.82 or \cite{Agliano}, p. 582, a congruence $\phi\in Con(A)- \{\nabla_A \}$  is ${\emph{prime}}$ if for all $\alpha, \beta \in Con(A)$, $[\alpha,\beta] \subseteq \phi$ implies $\alpha \subseteq \phi$ or $\beta \subseteq \phi$. It is easy to see that $\phi\in Con(A)- \{\nabla_A \}$  is a prime congruence if and only if for all $\alpha, \beta \in K(A)$, $[\alpha,\beta] \subseteq \phi$ implies $\alpha \subseteq \phi$ or $\beta \subseteq \phi$. Let us introduce the following notations: $Spec(A)$ is the set of prime congruences and $Max(A)$ is the set of maximal elements of $Con(A)$. If $\theta \in Con(A)- \{\nabla_A \}$  then there exists $\phi \in Max(A)$ such that $\theta \subseteq \phi$. By \cite{Agliano}, the following inclusion $Max(A) \subseteq Spec(A)$ holds.

According to \cite{Agliano}, p.582, the {\emph{radical}} $\rho(\theta)=\rho_A(\theta)$ of a congruence $\theta \in A$ is defined by $\rho_A(\theta)=\bigwedge \{\phi\in Spec(A)|\theta \subseteq \phi\}$; if $\theta=\rho(\theta)$ then $\theta$ is a radical congruence. We shall denote by $RCon(A)$ the set of radical congruences of $A$. The algebra $A$ is {\emph{semiprime}} if $\rho(\Delta_A)=\Delta_A$.

\begin{lema}
\cite{Agliano},\cite{GM2}
For all congruences $\alpha, \beta \in Con(A)$ the following hold:
\usecounter{nr}
\begin{list}{(\arabic{nr})}{\usecounter{nr}}
\item $\alpha \subseteq \rho(\alpha)$;
\item $\rho(\alpha \land \beta)=\rho ([\alpha,\beta])=\rho(\alpha) \land \rho(\beta)$;
\item $\rho(\alpha)= \nabla_A$ iff $\alpha = \nabla_A$;
\item $\rho(\alpha \lor \beta)=\rho(\rho(\alpha) \lor \rho(\beta))$;
\item $\rho(\rho(\alpha))=\rho(\alpha)$;
\item $\rho(\alpha) \lor \rho(\beta) = \nabla_A$ iff $\alpha \lor \beta = \nabla_A$;
\item $\rho([\alpha,\alpha]^n)=\rho(\alpha)$, for all integer $n \geq 0$.
\end{list}
\end{lema}

The equality $(4)$ from the previous lemma can be extended to arbitrary joins: for any family $(\alpha_i)_{i\in I} \subseteq A$, we have $\rho(\displaystyle \bigvee_{i \in I}\alpha_i)=\rho(\bigvee_{i \in I} \rho(\alpha_i))$. A new join operation can be introduced in $RCon(A)$: if $(\alpha_i)_{i\in I} \subseteq RCon(A)$ then we define $\displaystyle \bigvee_{i \in I}^{\cdot} \alpha_i=\rho(\bigvee_{i \in I}\alpha_i)$. An easy computation shows that $(RCon(A), \displaystyle \bigvee^{\cdot}, \wedge, \rho(\Delta_A), \nabla_A)$ is a frame (\cite{Johnstone} is the fundamental text for the frame theory). The frame $RCon(A)$ plays an important role in universal algebra: numerous problems on congruences of the algebra $A$ can be solved by using only the abstract frame structure of $RCon(A)$ (see \cite{Agliano}, \cite{Birkhoff}, \cite{Fresee}, \cite{GM2}, \cite{GKM}).

In what follows we shall identify the variety $\mathcal{V}$ with the category whose objects are algebras in $\mathcal{V}$ and the morphisms are the usual $\tau$ - homomorphisms (recall that $\tau$ is the signature of algebras in $\mathcal{V}$).

Let $u:A\rightarrow B$ be an arbitrary morphism in $\mathcal{V}$ and $u^*:Con(B)\rightarrow Con(A)$, $u^{\bullet}:Con(A)\rightarrow Con(B)$ are the maps defined by $u^*(\beta) = u^{-1}(\beta)$ and $u^{\bullet}(\alpha) = Cg_B(f(\alpha))$, for all $\alpha\in Con(A)$ and $\beta\in Con(B)$. Thus $u^{\bullet}$ is the left adjoint of $u^*$: for all $\alpha\in Con(A)$, $\beta\in Con(B)$, we have $u^{\bullet}(\alpha)\subseteq \beta$ iff  $\alpha\subseteq u^*(\beta)$.

By \cite{GKM}, if $\alpha\in K(A)$ then $u^{\bullet}(\alpha)\in K(B)$, so we can consider the restriction $u^{\bullet}|_{K(A)}:K(A)\rightarrow K(B)$.

We remark that the study of some themes in universal algebra needs to restrict the class of morphisms. For example, the investigation of functorial properties of reticulation imposed the class of admissible morphisms of the variety  $\mathcal{V}$ (see \cite{GKM}). According to \cite{GM1}, the morphism $u:A\rightarrow B$ of $\mathcal{V}$ is said to be admissible if $u^*(\psi)\in Spec(A)$, for any $\psi\in Spec(B)$. By Proposition 3.6 of \cite{GM1}, any surjective morphism of $\mathcal{V}$ is admissible.

Now we shall recall from \cite{Agliano} some notions regarding the topological structure of the prime spectrum of the algebra $A\in\mathcal{V}$.

For any congruence $\theta$ of $A$ we denote $V_A(\theta) = V(\theta) = \{\phi \in Spec(A)|\theta\subseteq \phi\}$ and $D_A(\theta) = D(\theta) = Spec(A)- V(\theta)$. If $\alpha, \beta \in Con(A)$ then $D(\alpha)\bigcap D(\beta) = D([\alpha,\beta])$ and $V(\alpha)\bigcup V(\beta) = V([\alpha,\beta])$. For any family of congruences $(\theta_i)_{i\in I}$, the following equalities hold: $\bigcup_{i\in I}D(\theta_i) = D(\bigvee_{i\in I}\theta_i)$ and $\bigcap_{i\in I}V(\theta_i) = V(\bigvee_{i\in I}\theta_i)$. Thus $Spec(A)$ becomes a topological space whose open sets are $D(\theta),\theta\in Con(A)$. We remark that this topology extends the Zariski topology (defined on the prime spectra of commutative rings) and the Stone topology (defined on the prime spectra of bounded distributive lattices). Thus this topology on the prime spectrum $Spec(A)$ of the algebra $A$ will be named Zariski topology and the respective topological space will be denoted by $Spec_Z(A)$. We mention that the family $(D(\alpha))_{\alpha\in K(A)}$ is a basis of open sets for the Zariski topology. We shall denote by $Spec_Z(A)$ the maximal spectrum $Max(A)$ endowed with the restriction of the topology of $Spec_Z(A)$.

The paper \cite{Agliano} contains a lot of results regarding the topological spaces $Spec_Z(A)$ and $Spec_Z(A)$.

\section{Reticulation of a semidegenerate congruence-modular algebra}

\hspace{0.5cm} Let us fix a semidegenerate congruence - modular variety $\mathcal{V}$ and $A$ an algebra of $\mathcal{V}$. Recall that the reticulation of an algebra $A\in \mathcal{V}$ is a bounded distributive lattice $L(A)$ such that the prime spectra $Spec_Z(A)$ and $Spec_{Id,Z}(L(A))$ are homeomorphic. By generalizing an idea of Belluce \cite{Belluce}, for each algebra $A\in \mathcal{V}$ we shall define a bounded distributive lattice $L(A)$. In general, the prime spectra $Spec_Z(A)$ and $Spec_{Id,Z}(L(A))$ are not homeomorphic. In this section we shall introduce the quasi-algebras in the variety $\mathcal{V}$, a class of algebras that fulfills the mentioned property.

Denote by $C(A)$ the subset of $Con(A)$ defined by the following rules:

$\bullet$ $K(A)\subseteq C(A)$;

$\bullet$ If $\theta,\chi\in C(A)$ then $\theta\lor \chi\in C(A)$;

$\bullet$ If $\theta,\chi\in C(A)$ then $[\theta, \chi]\in C(A)$.

\begin{remarca} If $A$ is a ring then $C(A)$ is exactly the commutative semi-ring $Sem(A)$, generated by the principal ideals of $A$, under the commutator operation and the sum (cf. \cite{Belluce}, p. 1856 or \cite{Belluce0}, p. 1515). In the general case of an algebra $A\in \mathcal{V}$, the algebraic structure $(C(A),\lor,[\cdot,\cdot], \Delta_A, \nabla_A)$ is similar to a semi-ring, but without the associativity of multiplication $[\cdot,\cdot]$.
\end{remarca}

Consider the following equivalence relation on $Con(A)$: for all $\alpha, \beta\in Con(A)$, $\alpha\equiv \beta$ if and only if $\rho(\alpha) = \rho(\beta)$. Let $\hat \alpha$ be the equivalence class of $\alpha\in Con(A)$ and $0 = \hat{\Delta_A}, 1 = \hat{\nabla_A}$. Then $\equiv$ is a congruence (on the lattice $Con(A)$) w.r.t. the join and the commutator operations: for all $\alpha,\beta\in Con(A)$, $\alpha\equiv \alpha'$ and $\beta\equiv \beta'$ implies $\alpha\lor\beta\equiv \alpha'\lor\beta' $ and  $[\alpha,\beta]\equiv [\alpha',\beta']$. Then the quotient set $L(A)$ = $C(A)/{\equiv}$ is a bounded distributive lattice. We shall denote by $\lambda_A:C(A)\rightarrow L(A)$ the function defined by $\lambda_A(\alpha) = \hat{\alpha}$, for all $\alpha\in C(A)$.

We remark that for all $\alpha,\beta\in C(A)$ we have $\lambda_A(\alpha) = \lambda_A(\beta)$ if and only if $\rho(\alpha) = \rho(\beta)$.

Our definition of the lattice $L(A)$ generalizes the construction given in \cite{Belluce} for a ring $R$; in this particular case, $C(A)$ is exactly the semiring $Sem(R)$ considered in the mentioned paper.

\begin{lema}
 For all congruences $\alpha, \beta \in C(A)$ the following hold:
\usecounter{nr}
\begin{list}{(\arabic{nr})}{\usecounter{nr}}
\item $\lambda_A(\alpha \lor \beta) = \lambda_A(\alpha)\lor \lambda_A(\beta)$;
\item $\lambda_A([\alpha,\beta])$ = $\lambda_A(\alpha)\land \lambda_A(\beta)$;
\item $\lambda_A(\alpha) = 1$ iff $\alpha = \nabla_A$;
\item If $\alpha\subseteq \beta$ then $\lambda_A(\alpha)\leq \lambda_A(\beta)$;
\item If $A$ is semiprime then $\lambda_A(\alpha) = 0$ iff $\alpha = \Delta_A$;
\item  $\lambda_A(\alpha)\leq \lambda_A(\beta)$ iff $\rho(\alpha)\subseteq\rho(\beta)$ iff for all $\phi\in Spec(A)$, $\beta\subseteq \phi$ implies $\alpha\subseteq \phi$.
\end{list}
\end{lema}

\begin{propozitie} Let $(\theta_j)_{j\in J}$ be a family of congruences in $C(A)$ such that $\bigvee_{j\in J}\theta_j \in C(A)$. Thus $\lambda_A(\bigvee_{j\in J}\theta_j)= \bigvee_{j\in J}\lambda_A(\theta_j)$.
\end{propozitie}

\begin{proof} Assume that $\theta\in C(A)$ and $\lambda_A(\theta_j)\leq \lambda_A(\theta)$, for all $j\in J$.  In order to obtain the desired equality, we have to prove that $\lambda_A(\bigvee_{j\in J}\theta_j)\leq \lambda_A(\theta)$.

Let $\phi$ be a prime congruence of $A$ such that $\theta\subseteq\phi$. According to Lemma 3.2(6), we have $\theta_j\subseteq \phi$, for all $j\in J$, therefore $\bigvee_{j\in J}\theta_j\leq \phi$. By using Lemma 3.2(4) we obtain $\lambda_A(\bigvee_{j\in J}\theta_j)\leq \lambda_A(\theta)$.
\end{proof}

Now we shall investigate (by means of the map $\lambda_A$) the connections between the congruences of algebra $A$ and the ideals of the lattice $L(A)$. Our main objective is to find a class of algebras $A$ in the variety $\mathcal{V}$ such that $Spec_Z(A)$ and $Spec_{Id,Z}(L(A))$ are homeomorphic spaces.

For all $\theta\in Con(A)$ and $I\in Id(L(A))$ we shall denote

$\theta^{\ast} = \{\lambda_A(\alpha)|\alpha\in K(A), \alpha\subseteq \theta \}$ and $I_{\ast} =\bigvee\{\alpha\in K(A)|\lambda_A(\alpha)\in I\}$.

Thus $\theta^{\ast}$ is an ideal of the lattice $L(A)$ and $I_{\ast}$ is a congruence of $A$. In this way one obtains two order - preserving functions $(\cdot)^{\ast}:Con(A)\rightarrow Id(L(A))$ and $(\cdot)_{\ast}:Id(L(A))\rightarrow Con(A)$. These two functions will play an important role in transferring some properties from congruences of $A$ to ideals of $L(A)$ and vice-versa.

\begin{lema} If $\theta\in C(A)$ then $\theta^{\ast}=(\lambda_A(\theta)]$, where the second member is the principal lattice ideal generated by $\{\lambda_A(\theta)\}$ in $L(A)$.
\end{lema}

\begin{proof}
From $\lambda_A(\theta)\in \theta^{\ast}$ we obtain $(\lambda_A(\theta)]\subseteq \theta^{\ast}$. In order to prove the converse inclusion $\theta^{\ast}\subseteq (\lambda_A(\theta)]$ consider an element $x\in \theta^{\ast}$, so $x=\lambda_A(\alpha)$ for some $\alpha\in C(A)$  such that $\alpha\subseteq \theta$. Thus $x=\lambda_A(\alpha)\leq \lambda_A(\theta)$, so $x\in (\lambda_A(\theta)]$.
\end{proof}

\begin{lema} For all $\alpha\in K(A)$ and $I\in Id(L(A))$, $\alpha\subseteq I_{\ast}$ if and only if $\lambda_A(\alpha)\in I$.
\end{lema}

\begin{proof} Assume that $\alpha\subseteq I_{\ast}=\bigvee\{\alpha\in K(A)|\lambda_A(\alpha)\in I\}$ so there exist an integer $\geq 1$ and $\beta_1,\cdots,\beta_n\in K(A)$ such that $\alpha\subseteq \beta_1\lor...\lor \beta_n$ and $\lambda_A(\beta_i)\in I$, for all $i=1,\cdots,n$ (because $\alpha$ is a compact congruence). By applying Lemma 3.2, (1) and (4) we get $\lambda_A(\alpha)\leq \lambda_A(\beta)\lor...\lor\lambda_A(\beta_n)\in I$, so $\lambda_A(\alpha)\in I$ (because $I$ is an ideal of $L(A)$). The converse implication is obvious.
\end{proof}

\begin{lema} If $\theta\in Con(A)$ and $I\in Id(L(A))$ then $\theta\subseteq (\theta^{\ast})_{\ast}$ and $I\subseteq(I_{\ast})^{\ast}$.
\end{lema}

\begin{proof} Remark that $(\theta^{\ast})_{\ast}= \bigvee\{\alpha\in K(A)|\lambda_A(\alpha)\in \theta^{\ast}\}$. We observe that $\alpha\in K(A)$ and $\alpha\subseteq \theta$ imply that $\lambda_A(\alpha)\in \theta^{\ast}$, so $\alpha\subseteq (\theta^{\ast})_{\ast}$. Then the inclusion $\theta\subseteq (\theta^{\ast})_{\ast}$ follows.

In order to prove that $I\subseteq(I_{\ast})^{\ast}$, assume that $x\in I$, so $x=\lambda_A(\varepsilon)$ for some $\varepsilon\in C(A)$. Let $\alpha$ be a compact congruence of $A$ such that $\alpha\subseteq \varepsilon$, hence $\lambda_A(\alpha)\leq \lambda_A(\varepsilon)$. Thus $\lambda_A(\alpha)\in I$, hence, by using Lemma 3.5, one obtains $\alpha\subseteq I_{\ast}$. It follows that $\varepsilon\subseteq I_{\ast}$, so $x=\lambda_A(\varepsilon)\in (I_{\ast})^{\ast}$. We conclude that $I\subseteq(I_{\ast})^{\ast}$.

\end{proof}

\begin{lema}If $\phi\in Spec(A)$ then $(\phi^{\ast})_{\ast}=\phi$.
\end{lema}

\begin{proof} In order to prove that $(\phi^{\ast})_{\ast}\subseteq \phi$, let us consider a compact congruence $\alpha$ such that $\alpha\subseteq (\phi^{\ast})_{\ast}$, therefore $\lambda_A(\alpha)\in \phi^{\ast}$ (cf. Lemma 3.5). Thus $\lambda_A(\alpha)=\lambda_A(\beta)$ for some $\beta\in C(A)$ such that $\beta\subseteq\phi$. From $\rho(\alpha)=\rho(\beta)$, $\beta\subseteq\phi$ and $\phi\in Spec(A)$ it follows that $\alpha\subseteq \phi$. Then the inclusion $(\phi^{\ast})_{\ast}\subseteq \phi$ is proven. The converse inclusion $\phi\subseteq (\phi^{\ast})_{\ast}$ is assured by Lemma 3.4.
\end{proof}

\begin{lema} An ideal $I$ of the lattice $L(A)$ is proper if and only if $I_{\ast}$ is a proper congruence.
\end{lema}

\begin{proof} Assume $I_{\ast}=\nabla_A$, so $\nabla_A=\bigvee\{\alpha\in K(A)|\lambda_A(\alpha)\in I\}$. Since $\nabla_A$ is a compact congruence it follows that there exist $\alpha_1,\cdots,\alpha_n\in K(A)$ such that $\alpha_1\lor...\lor \alpha_n =\nabla_A$ and $\lambda_A(\alpha_i)\in I$, for $i=1,\cdots,n$. Thus $1=\lambda_A(\alpha_1)\lor...\lor\lambda_A(\alpha_n)\in I$, i.e. $I=L(A)$. The converse implication is obvious.
\end{proof}

\begin{lema} A congruence $\theta$ of $A$ is proper if and only if the ideal $\theta^{\ast}$ of the lattice $L(A)$ is proper.
\end{lema}

\begin{proof} Assume that $1\in \theta^{\ast}$ so there exists $\alpha\in C(A)$ such that $\alpha\subseteq\theta$ and $\lambda_A(\alpha)=1$. By Lemma 3.2(3) we get $\alpha=\nabla_A$, hence $\theta=\nabla_A$.  The converse implication is obvious.
\end{proof}

Let $L$ be a bounded distributive lattice and $Id(L)$ the set of its ideals. Then $Spec_{Id}(L)$ will denote the set of prime ideals in $L$ and $Max_{Id}(L)$ the set of maximal ideals in $L$. $Spec_{Id}(L)$ (resp. $Max_{Id}(L)$) endowed with Stone topology will be denoted by $Spec_{Id,Z}(L)$ (resp. $Max_{Id,Z}(L)$).

For any ideal $I$ of $L$ we denote $D_{Id}(I) = \{Q\in Spec_{Id}(L)|I\not\subseteq Q\}$ and $V_{Id}(I) = \{Q\in Spec_{Id}(L)|I\subseteq Q\}$. If $x\in L$ then we use the notation $D_{Id}(x) = D_{Id}((x])  = \{Q\in Spec_{Id}(L)|x\notin Q\}$ and $V_{Id}(x) = V_{Id}((x])  = \{Q\in Spec_{Id}(L)|x\in Q\}$, where $(x]$ is the principal ideal of $L$ generated by the set $\{x\}$. Recall from \cite{BalbesDwinger}, \cite{Johnstone} that the family $(D_{Id}(x))_{x\in L}$ is a basis of open sets for the Stone topology on $Spec_{Id}(L)$.

\begin{lema} If $\phi\in Spec(A)$ then $\phi^{\ast}$ is a prime ideal of the lattice $L(A)$.
\end{lema}

\begin{proof} By Lemma 3.9, $\phi^{\ast}$ is a proper ideal of $L(A)$. Assume that $x, y$ are two elements of $L(A)$ such that $x\land y\in \phi^{\ast}$. One can find $\alpha,\beta\in C(A)$ such that $x=\lambda_A(\alpha)$, $y=\lambda_A(\beta)$, hence $\lambda_A([\alpha,\beta])=x\land y \in\phi^{\ast}$. Then $\lambda_A([\alpha,\beta])=\lambda_A(\gamma)$, for some $\gamma\in C(A)$ such that $\gamma\subseteq \phi$. Thus we get $\rho([\alpha,\beta])=\rho(\gamma)$, hence $[\alpha,\beta]\subseteq\phi$, because $\gamma\subseteq \phi$ and $\phi\in Spec(A)$. It follows that $\alpha\subseteq \phi$ or $\beta\subseteq \phi$, so we get $x=\lambda_A(\alpha)\in \phi^{\ast}$ or  $y=\lambda_A(\beta)\in \phi^{\ast}$. Conclude that the ideal $\phi^{\ast}$ is  prime.
\end{proof}

By the previous lemma, one can consider the map $u:Spec(A)\rightarrow Spec_{Id}(L(A))$ defined by $u(\phi)=\phi^{\ast}$, for any $\phi\in Spec(A)$.

\begin{lema}
\usecounter{nr}
\begin{list}{(\arabic{nr})}{\usecounter{nr}}
\item For any ideal $I$ of $L(A)$ we have $u^{-1}(D_{Id}(I))= D(I_{\ast})$;
\item $u$ is an injective continuous function.
\end{list}
\end{lema}

\begin{proof}
$(1)$ Assume that $I$ is an ideal of the lattice $L(A)$. We shall prove that for any $\phi\in Spec(A)$, the following equivalence holds: $I\subseteq \phi^{\ast}$ if and only if $I_{\ast}\subseteq \phi$.

Assuming that $I\subseteq \phi^{\ast}$ it follows that $I_{\ast}\subseteq (\phi^{\ast})_{\ast}= \phi$ (by Lemma 3.7). Conversely, if $I_{\ast}\subseteq \phi$ then $I\subseteq (I_{\ast})^{\ast}\subseteq \phi^{\ast}$ (by Lemma 3.6). Then the desired equivalence is proven.

It follows that for any $\phi\in Spec(A)$, the following equivalences hold:

$\phi\in u^{-1}(D_{Id}(I))$ iff $\phi^{\ast}\in D_{Id}(I)$ iff $I\not\subseteq \phi^{\ast}$ iff $\phi\in D(I_{\ast})$,

hence the equality $u^{-1}(D_{Id}(I))= D(I_{\ast})$ is checked.

$(2)$ Assume that $\phi_1,\phi_2\in Spec(A)$ and $\phi_1^{\ast}=\phi_2^{\ast}$, hence, by applying Lemma 3.7 we get $\phi_1= (\phi_1^{\ast})_{\ast}=(\phi_2^{\ast})_{\ast} =\phi_2$. Thus $u$ is injective. That $u$ is a continuous function follows by $(1)$.

\end{proof}

\begin{propozitie}
The following assertions are equivalent
\usecounter{nr}
\begin{list}{(\arabic{nr})}{\usecounter{nr}}
\item  For any ideal $I$ of $L(A)$, we have $I=(I_{\ast})^{\ast}$;
\item  For any prime ideal $P$ of $L(A)$, we have $P_{\ast}\in Spec(A)$.
\end{list}
\end{propozitie}

\begin{proof}
$(1) \Rightarrow (2)$ Let $P$ be a prime ideal of $L(A)$.  Assume that $\alpha,\beta\in K(A)$ and $[\alpha,\beta]\subseteq P_{\ast}$. Then $[\alpha,\beta]\in C(A)$ and $\lambda_A(\alpha)\land \lambda_A(\beta)= \lambda_A([\alpha,\beta])\in (P_{\ast})^{\ast}= P$, so $\lambda_A(\alpha)\in P$ or $\lambda_A(\beta)\in P$. By applying Lemma 3.5, one gets $\alpha\subseteq P_{\ast}$ or $\beta\subseteq P_{\ast}$, so $P_{\ast}\in Spec(A)$.

$(2) \Rightarrow (1)$ Firstly we shall prove that for any prime ideal $P$ of $L(A)$ we have $P=(P_{\ast})^{\ast}$. According to Lemma 3.6 it suffices to show that $(P_{\ast})^{\ast}\subseteq P$. By hypothesis we have $P_{\ast}\in Spec(A)$. In order to establish the inclusion $(P_{\ast})^{\ast}\subseteq P$ it suffices to prove the following implication:

$\theta\in C(A), \theta\subseteq P_{\ast}$ $\Rightarrow $ $\lambda_A(\theta)\in P$.

We shall prove this implication by induction on how the set $C(A)$ is defined:

$\bullet$ If $\theta\in K(A)$ then the implication follows by Lemma 3.5.

$\bullet$ Assume that $\theta= \theta_1\lor \theta_2$ and $\theta_1, \theta_2\in C(A)$ satisfy the hypothesis of induction. From $\theta\subseteq P_{\ast}$ it follows that $\theta_1\subseteq P_{\ast}$ and $\theta_2\subseteq P_{\ast}$. Thus $\lambda_A(\theta_1)\in P$ and $\lambda_A(\theta_2)\in P$ so $\lambda_A(\theta)= \lambda_A(\theta_1)\lor \lambda_A(\theta_2)\in P$.

$\bullet$ Assume that $\theta= [\theta_1,\theta_2]$ and $\theta_1, \theta_2\in C(A)$ satisfy the hypothesis of induction. From $\theta\subseteq P_{\ast}$ and $P_{\ast}\in Spec(A)$ it follows that $\theta_1\subseteq P_{\ast}$ or $\theta_2\subseteq P_{\ast}$. By hypothesis of induction we obtain $\lambda_A(\theta_1)\in P$ or $\lambda_A(\theta_2)\in P$, hence $\lambda_A(\theta)= \lambda_A(\theta_1)\land \lambda_A(\theta_2) \in P$.

Consider now an arbitrary ideal $I$ of $L(A)$. We shall prove that $(I_{\ast})^{\ast}\subseteq I$. Let $P$ be a prime ideal of $L(A)$ such that $I\subseteq P$. According to the first part of the proof we have $(I_{\ast})^{\ast}\subseteq (P_{\ast})^{\ast}=P$. Therefore $(I_{\ast})^{\ast}\subseteq P$ for all prime ideals $L$ of $L(A)$ such that $I\subseteq P$, so $(I_{\ast})^{\ast}\subseteq \bigcap\{P\in Spec_{Id}(L(A))|I\subseteq P\}= I$. By applying Lemma 3.6 we get $I=(I_{\ast})^{\ast}$.

\end{proof}

\begin{propozitie}
Assume that the equivalent conditions from Proposition 3.12 are verified. Then $(\cdot)^{\ast}$ is the left adjoint of $(\cdot)_{\ast}$.
\end{propozitie}

\begin{proof} We have to show that for all $\theta\in Con(A)$ and $I\in Id(L(A))$ we have $\theta^{\ast}\subseteq I$ if and only if $\theta\subseteq I_{\ast}$. If $\theta^{\ast}\subseteq I$ then $\theta\subseteq (\theta^{\ast})_{\ast}\subseteq I_{\ast}$ (by Lemma 3.6). Conversely, if $\theta\subseteq I_{\ast}$, then by applying that $I=(I_{\ast})^{\ast}$, we obtain the inclusions $\theta^{\ast}\subseteq (I_{\ast})^{\ast}= I$.
\end{proof}

\begin{lema} Assuming that the equivalent conditions from Proposition 3.12 are verified the following assertions hold
\usecounter{nr}
\begin{list}{(\arabic{nr})}{\usecounter{nr}}
\item If $\theta,\chi\in Con(A)$ then $[\theta,\chi]^{\ast}$ = $(\theta\land \chi)^{\ast}$ = $\theta^{\ast}\bigcap \chi^{\ast}$;
\item If $(\theta_i)_{i\in I}$ is a family of congruences of $A$ then $(\bigvee_{i\in I}\theta_i)^{\ast} = \bigvee_{i\in I}\theta_i^{\ast}$;
\item If $(I_t)_{t\in T}$ is a family of ideals of $L(A)$ then $(\bigcap_{t\in T}I_t)_{\ast} = \bigcap_{t\in T}(I_t)_{\ast}$.

\end{list}
\end{lema}

\begin{proof} $(1)$ Since $[\theta,\chi]\subseteq \theta\cap\chi\subseteq \theta$ and $[\theta,\chi]\subseteq \theta\cap\chi\subseteq \chi$ and $(\cdot)^{\ast}$ is order-preserving we obtain the inclusions $[\theta,\chi]^{\ast}\subseteq (\theta\cap\chi)^{\ast}\subseteq \theta^{\ast}\cap\chi^{\ast}$.

If $x\in \theta^{\ast}\bigcap \chi^{\ast}$ then there exist $\alpha,\beta\in C(A)$ such that $x=\lambda_A(\alpha)=\lambda_A(\beta)$, $\alpha\subseteq \theta$ and $\beta\subseteq \chi$. Thus $x=\lambda_A(\alpha)\land \lambda_A(\beta)= \lambda_A([\alpha,\beta])\in [\theta,\chi]^{\ast}$, because $[\alpha,\beta]\in C(A)$ and $[\alpha,\beta]\subseteq [\theta,\chi]$. It follows that $\theta^{\ast}\bigcap \chi^{\ast}\subseteq [\theta,\chi]^{\ast}$, therefore $[\theta,\chi]^{\ast}$ = $(\theta\land \chi)^{\ast}$ = $\theta^{\ast}\bigcap \chi^{\ast}$.

$(2)$ and $(3)$ follow by applying the adjointness situation described in Proposition 3.13.

\end{proof}

Assuming that the equivalent conditions from  Proposition 3.12 are verified one can consider the order - preserving function $v:Spec_{Id}L((A))\rightarrow Spec(A)$, defined by  $v(P) = P_{\ast}$, for all $P\in Spec_{Id}(L(A))$.

\begin{lema} Assuming that the equivalent conditions from Proposition 3.12 are verified we have $v^{-1}(D(\theta))= D_{Id}(\theta^{\ast})$, for any congruence $\theta$ of $A$;

\end{lema}

\begin{proof} Assume that $\theta$ is a congruence of $A$ and $P$ is a prime ideal of $L(A)$. Firstly, we shall prove the following equivalence: $\theta^{\ast}\subseteq P$ if and only if $\theta\subseteq P_{\ast}$. By using Lemma 3.6,  $\theta^{\ast}\subseteq P$ implies $\theta\subseteq (\theta^{\ast})_{\ast}\subseteq P_{\ast}$. Conversely, $\theta\subseteq P_{\ast}$ implies $\theta^{\ast}\subseteq (P_{\ast})^{\ast}= P$ (the last equality follows from the hypothesis that the equivalent conditions from Proposition 3.12 are verified). Then for each prime ideal $P$ of $L(A)$, $\theta^{\ast}\subseteq P$ if and only if $\theta\subseteq P_{\ast}$, hence the following equivalences hold: $P\in v^{-1}(D(\theta))$ iff $P_{\ast}\in D(\theta)$ iff $\theta^{\ast}\not\subseteq P$ iff $\theta\not\subseteq P_{\ast}$ iff $P\in D_{Id}(\theta^{\ast})$. Conclude that $v^{-1}(D(\theta))= D_{Id}(\theta^{\ast})$.

\end{proof}

\begin{propozitie}
If the equivalent conditions from Proposition 3.12 are verified then $u:Spec_Z(A)\rightarrow Spec_{Id,Z}(L(A))$ and $v:Spec_{Id,Z}(L(A))\rightarrow Spec_Z(A)$ are homeomorphisms, inverse to one another.
\end{propozitie}

\begin{proof} By using Proposition 3.12(1) and Lemma 3.7, for all $P\in Spec_{Id}(L(A))$ and $\phi\in Spec(A)$, we have $u(v(P))= (P_{\ast})^{\ast}=P$ and $v(u(\phi))=(\phi^{\ast})_{\ast}=\phi$, so $u$ and $v$ are bijective functions, inverse to one another. That $u$ and $v$ are continuous functions follows by applying Lemmas 3.11 and 3.15.
\end{proof}

According to the previous result, if the equivalent conditions of Proposition 3.12 are verified then the prime spectra of $A$ and $L(A)$ are homeomorphic, so we can say that $L(A)$ is the reticulation of the algebra $A$.

Recall from \cite{Hochster} that a topological space $X$ is a spectral space if it fulfills the following properties:

$\bullet$ $X$ is a compact $T_0$-space;

$\bullet$ the open and compact subsets of $X$ form a basis of the topology, closed under finite intersections;

$\bullet$ any closed and irreducible subset of $X$ has a generic point.

The prime spectrum $Spec_Z(R)$ of a commutative ring $R$ and the prime spectrum $Spec_{Id,Z}(L)$ of a bounded distributive lattice $L$ are the main examples of spectral spaces (see \cite{Dickmann}, \cite{Johnstone}).  Let $L$ be a bounded distributive lattice and $\mathcal{B}$ the basis of compact open subsets of $Spec_{Id,Z}(L)$. By the Stone duality of bounded distributive lattices we know that $\mathcal{B}$ = $(D_{Id}(a))_{a\in L}$ (see \cite{Belluce}).

\begin{corolar} Assume that the equivalent conditions from Proposition 3.12 are verified. Then $Spec_Z(A)$ is a spectral space and $(D(\alpha))_{\alpha\in C(A)}$ is the basis of compact open subsets of $Spec_Z(A)$.
\end{corolar}

\begin{proof} We know that $(D_{Id}(\lambda_A(\alpha)))_{\alpha\in C(A)}$ is the basis of compact sets of $Spec_{Id,Z}(L(A))$. According to Proposition 3.16, $(u^{-1}(D_{Id}(\lambda_A(\alpha))))_{\alpha\in C(A)}$ is the basis of compact open sets of $Spec_(A)$. By Lemma 3.4, for any $\alpha\in C(A)$ we have $D_{Id}(\lambda_A(\alpha))=D_{Id}(\lambda_A((\alpha)])= D_{Id}(\alpha^{\ast})$. Thus, by using Lemma 3.11(1), we get $u^{-1}(D_{Id}(\lambda_A(\alpha)))=u^{-1}(D_{Id}(\alpha^{\ast}))= D((\alpha^{\ast})_{\ast})$.

We remark that for any $\phi\in Spec(A)$ we have: $\alpha\subseteq \phi$ if and only if $(\alpha^{\ast})_{\ast}\subseteq \phi$. Indeed, $\alpha\subseteq \phi$  implies $ (\alpha^{\ast})_{\ast}\subseteq (\phi^{\ast})_{\ast}= \phi$ (by Lemma 3.7) and the converse implication follows by using Lemma 3.6. Therefore $V(\alpha)= V((\alpha^{\ast})_{\ast})$, hence  $u^{-1}(D_{Id}(\lambda_A(\alpha)))= D((\alpha^{\ast})_{\ast})= D(\alpha)$. It follows that  $(D(\alpha))_{\alpha\in C(A)}$ is the basis of compact open subsets of $Spec_Z(A)$.
\end{proof}

\begin{propozitie} If $A$ is an arbitrary algebra in the semidegenerate congruence-modular variety $\mathcal{V}$ then the following hold
\usecounter{nr}
\begin{list}{(\arabic{nr})}{\usecounter{nr}}
\item If $\phi\in Max(A)$ then $\phi^{\ast}\in Max_{Id}(L(A))$;
\item If $M\in Max_{Id}(L(A))$ then $M_{\ast}\in Max(A)$;
\item $u|_{Max(A)}:Max_Z(A)\rightarrow Max_{Id,Z}(L(A))$ is a bijective map;
\item The following two functions $u|_{Max(A)}:Max_Z(A)\rightarrow Max_{Id,Z}(L(A))$ and $v|_{Max_{Id}(L(A))}: Max_{Id,Z}(L(A))\rightarrow Max_Z(A)$ are homeomorphisms, inverse to one another.
\item  Then $Max_Z(A)$ is a compact $T_1$-space.
\end{list}
\end{propozitie}

\begin{proof}
$(1)$ Assume that $\phi\in Max(A)$, so $\phi^{\ast}\in Spec_{Id}(L(A))$ (cf. Lemma 3.10). Let $I$ be a proper ideal of the lattice $L(A)$ such that $\phi^{\ast}\subseteq I$, so $\phi\subseteq(\phi^{\ast})_{\ast}\subseteq I_{\ast}$. By Lemma 3.8, $I_{\ast}$ is a proper congruence of $A$, therefore $\phi=I_{\ast}$. Then $I\subseteq (I_{\ast})^{\ast}=\phi^{\ast}\subseteq I$, hence $I=(I_{\ast})^{\ast}=\phi^{\ast}$. It follows that $\phi^{\ast}\in Max_{Id}(L(A))$.

$(2)$ Assume that $M\in Max_{Id}(L(A))$. Let $\theta$ be a proper congruence of $A$ such that $M_{\ast}\subseteq \theta$, so $M\subseteq (M_{\ast})^{\ast}\subseteq \theta^{\ast}$ and $\theta^{\ast}$ is a proper ideal of $L(A)$ (cf. Lemma 3.8). It results that $M=\theta^{\ast}$, hence $\theta\subseteq (\theta^{\ast})_{\ast}\subseteq M_{\ast}\subseteq \theta$, therefore $\theta=(\theta^{\ast})_{\ast}=M_{\ast}$. Then $M_{\ast}$ is a maximal congruence of $A$.

$(3)$ By $(1)$ and $(2)$.

$(4)$ By a routine verification.

$(5)$ We know from \cite{Johnstone}, p.66 that $ Max_{Id,Z}(L(A))$ is a compact $T_1$-space, then we apply Corollary (4).
\end{proof}

\begin{definitie} The algebra $A $ is said to be quasi-commutative if for all $\alpha,\beta\in PCon(A)$ there exists $\gamma\in K(A)$ such that $\gamma\subseteq [\alpha,\beta]$
and $\rho(\gamma)=\rho([\alpha,\beta])$.
\end{definitie}

The quasi-commutative algebras generalize the quasi-commutative rings, introduced by Belluce in \cite{Belluce}.

\begin{lema} The following are equivalent:
\usecounter{nr}
\begin{list}{(\arabic{nr})}{\usecounter{nr}}
\item $A$ is a quasi-commutative algebra;
\item For  all $\alpha,\beta\in K(A)$ there exists $\gamma\in K(A)$ such that $\gamma\subseteq [\alpha,\beta]$
and $\rho(\gamma)=\rho([\alpha,\beta])$.
\end{list}
\end{lema}

\begin{proof}
$(1) \Rightarrow (2)$ Assume that $\alpha,\beta\in K(A)$ hence there exist two finite sets $I,J$ and two families $(\alpha_i)_{i\in I}$, $(\beta_j)_{j\in J}$ of principal congruences such that  $\alpha=\bigvee_{i\in I}\alpha_i$, $\beta=\bigvee_{j\in J}\beta_j$. By using the hypothesis $(1)$ for any pair $(\alpha_i,\beta_j)$, one can find a family $(\gamma_{ij})_{i\in I,j\in J}$ of compact congruences such that $\gamma_{ij}\subseteq [\alpha_i,\beta_j]$ and $\rho(\gamma_{ij})=\rho([\alpha_i,\beta_j])$, for all $i\in I$ and $j\in J$. Denoting $\gamma=\bigvee_{i\in I,j\in J}[\alpha_i,\beta_j]$, it follows that $\gamma\in K(A)$, $\gamma\subseteq[\alpha,\beta]$ and $\rho(\gamma)=\rho(\bigvee_{i\in I,j\in J}\gamma_{ij})$ = $\rho(\bigvee_{i\in I,j\in J}\rho(\gamma_{ij}))$ = $\rho(\bigvee_{i\in I,j\in J}\rho([\alpha_i,\beta_j]))$ = $\rho(\bigvee_{i\in I,j\in J}[\alpha_i,\beta_j])=\rho([\alpha,\beta])$.

$(2)\Rightarrow (1)$ Obviously.
\end{proof}

\begin{lema} Assume $A$ is a quasi-commutative algebra. If $\theta\in C(A)$ then there exists $\gamma\in K(A)$ such that $\gamma\subseteq\theta$ and $ \rho(\theta)=\rho(\gamma)$.
\end{lema}

\begin{proof} Assume that $\theta\in C(A)$. We shall prove that $\theta$ fulfills the mentioned property by induction on the way in which the set $C(A)$ is defined :

$\bullet$ If $\theta\in K(A)$ then one applies Lemma 3.20 for $\theta,\nabla_A\in K(A)$: there exists $\gamma\in K(A)$ such that $\gamma\subseteq [\theta,\nabla_A]= \theta$
and $\rho(\gamma)=\rho([\theta,\nabla_A])= \rho(\theta)$.

$\bullet$ Assume that $\theta= \theta_1\lor \theta_2$ and $\theta_1,\theta_2\in C(A)$ satisfy the hypothesis of induction, so there exist the compact congruences $\gamma_1,\gamma_2$ such that $\gamma_i\subseteq \theta_i$ and $\rho(\gamma_i)= \rho(\theta_i)$, for $i=1,2$.
 Denoting $\gamma= \gamma_1\lor\gamma_2$ we get $\gamma\in K(A)$, $\gamma\subseteq \theta$ and $\rho(\gamma)= \rho(\gamma_1\lor \gamma_2)= \rho(\rho(\gamma_1)\lor \rho( \gamma_2))= \rho(\rho(\theta_1)\lor \rho(\theta_2))= \rho(\theta_1\lor \theta_2)=\rho(\theta)$.

$\bullet$ Assume that $\theta= [\theta_1,\theta_2]$ and $\theta_1,\theta_2\in C(A)$ satisfy the hypothesis of induction, so there exist the compact congruences $\gamma_1,\gamma_2$ such that $\gamma_i\subseteq \theta_i$ and $\rho(\gamma_i)= \rho(\theta_i)$, for $i=1,2$. Since $A$ is a quasi-commutative algebra, by applying Lemma 3.20 for $\gamma_1,\gamma_2\in K(A)$ one can find a compact element $\gamma$ such that $\gamma\subseteq [\gamma_1,\gamma_2]$ and $\rho(\gamma)= \rho( [\gamma_1,\gamma_2])$.

Thus $\gamma\subseteq [\gamma_1,\gamma_2]\subseteq [\theta_1,\theta_2]=\theta$ and $\rho(\gamma)= \rho( [\gamma_1,\gamma_2])= \rho(\gamma_1)\land \rho(\gamma_2) = \rho(\theta_1)\land \rho(\theta_2)= \rho( [\theta_1,\theta_2])= \rho(\theta)$, so the third step of induction is finished.

\end{proof}

\begin{propozitie} Assume $A$ is a quasi-commutative algebra. Then for each ideal $I$ of lattice $L(A)$ we have $I=(I_{\ast})^{\ast}$.
\end{propozitie}

\begin{proof} According to Lemma 3.6, it suffices to prove that $(I_{\ast})^{\ast}\subseteq I$. Assume that $x\in(I_{\ast})^{\ast}$, so $x=\lambda_A(\theta)$, for some $\theta\in C(A)$ such that $\theta\subseteq I_{\ast}$.

By using Lemma 3.21 we find a compact congruence $\gamma$ such that $\gamma\subseteq\theta$ and $\rho(\gamma)=\rho(\theta)$, hence $\lambda_A(\gamma)=\lambda_A(\theta)$. According to Lemma 3.5, $\gamma\subseteq\theta\subseteq I_{\ast}$ implies $\lambda_A(\gamma)\in I$, therefore $x=\lambda_A(\theta)=\lambda_A(\gamma)\in I$.
\end{proof}

\begin{lema}
 Assume that for each ideal of lattice $L(A)$ we have $I=(I_{\ast})^{\ast}$.  Then $V(\theta)=V((\lambda_A(\theta)]_{\ast})$, for each $\theta\in C(A)$.
\end{lema}

\begin{proof}
 Let $\phi$ be an arbitrary prime congruence of $A$. If $\theta\subseteq \phi$ then by using Lemmas 3.4 and 3.7 we get: $(\lambda_A(\theta)]_{\ast}= (\theta^{\ast})_{\ast}\subseteq (\phi^{\ast})_{\ast}=\phi$. It follows that $V(\theta)\subseteq V((\lambda_A(\theta)]_{\ast})$.

Conversely, assume that $(\lambda_A(\theta)]_{\ast}\subseteq \phi$, therefore, by using Lemmas 3.6 and 3.4 we obtain: $\theta\subseteq (\theta^{\ast})_{\ast}= (\lambda_A(\theta)]_{\ast} \subseteq \phi$. It follows that $V((\lambda_A(\theta)]_{\ast})\subseteq V(\theta)$.
\end{proof}

The spectral algebras, introduced by the following definition, generalize the spectral rings, defined by Belluce in \cite{Belluce}.

\begin{definitie} The algebra $A\in \mathcal{V}$ is said to be a spectral algebra if the following conditions are fulfilled:
\usecounter{nr}
\begin{list}{(\arabic{nr})}{\usecounter{nr}}
\item $Spec_Z(A)$ is a spectral space;
\item For any compact congruence $\alpha$,  $D(\alpha)$ is a compact subset of $Spec_Z(A)$ .

\end{list}
\end{definitie}

\begin{lema} If $A$ is a spectral algebra then for each $\theta\in C(A)$, $D(\theta)$ is a compact subset of $Spec_Z(A)$.
\end{lema}

\begin{proof} Assume that $A$ is a spectral algebra so $Spec_Z(A)$ is a spectral space. Thus the family  $\mathcal{B}$ of compact open sets is a basis for  $Spec_Z(A)$ which is closed under finite intersections. Suppose that $\theta\in C(A)$. We shall prove that $D(\theta)$ is compact by induction on how $C(A)$ is defined:

$\bullet$ Assume that $\theta\in K(A)$. According to the condition $(2)$ of Definition 3.24, $D(\theta)$ is compact.

$\bullet$ Assume that $\theta= \theta_1\lor \theta_2$ and $D(\theta_1), D(\theta_2)$ are compact subsets of $Spec_Z(A)$. Then $D(\theta)=D(\theta_1)\cup D(\theta_2)$ is compact.

$\bullet$  Assume that $\theta= [\theta_1, \theta_2]$ and $D(\theta_1), D(\theta_2)$ are compact subsets of $Spec_Z(A)$. Then $D(\theta)=D(\theta_1)\cap D(\theta_2)$ is compact (because $\mathcal{B}$ is closed under finite intersections).
\end{proof}

The following result extends Theorem 4 of \cite{Belluce} to a universal algebra framework.

\begin{teorema} The following are equivalent:
\usecounter{nr}
\begin{list}{(\arabic{nr})}{\usecounter{nr}}
\item $A$ is a spectral algebra;
\item $A$ is a quasi-commutative algebra;
\item For any ideal $I$ of $L(A)$, we have $I=(I_{\ast})^{\ast}$.
\end{list}
\end{teorema}

\begin{proof}
$(1)\Rightarrow (2)$ Assume that $A$ is a spectral algebra. Then $Spec_Z(A)$ is a spectral space, so  the family  $\mathcal{B}$ of compact open sets is a basis for  $Spec_Z(A)$ which is closed under finite intersections. In order to prove that $A$ is a quasi-commutative algebra let us consider two compact congruences $\alpha,\beta$ of $A$. By hypothesis, $D([\alpha,\beta])= D(\alpha)\cap D(\beta)\in \mathcal{B}$, so $D([\alpha,\beta])$ is a compact subset of $Spec_Z(A)$. Since $[\alpha,\beta]= \bigvee\{\gamma\in K(A)|\gamma\leq [\alpha,\beta]\}$, the following equality holds:

$D([\alpha,\beta])= \bigcup\{D(\gamma)|\gamma\in K(A),\gamma\leq [\alpha,\beta]\}$.

Thus there exist the compact congruences $\gamma_1,\cdots,\gamma_n$ such that $\gamma_i\subseteq [\alpha,\beta]$, for all $i=1,\cdots,n$ and $D([\alpha,\beta])=\bigcup_{i=1}^nD(\gamma_i)= D(\bigvee_{i=1}^n\gamma_i)$.

We set $\gamma= \bigvee_{i=1}^n\gamma_i$, hence $\gamma\in K(A)$, $\gamma\subseteq [\alpha,\beta]$ and $D([\alpha,\beta])= D(\gamma)$. By using the last equality we get $\rho([\alpha,\beta])=\bigcap V([\alpha,\beta])= \bigcap V(\gamma)= \rho(\gamma)$. According to Lemma 3.20, $A$ is a quasi-commutative algebra.

$(2)\Rightarrow (3)$ We apply Proposition 3.22.

$(3)\Rightarrow (1)$ According to Proposition 3.16, one can consider the homeomorphism $u:Spec_Z(A)\rightarrow Spec_{Id,Z}(L(A))$, defined by $u(\phi)=\phi^{\ast}$, for all $\phi \in Spec(A)$. We know from the Stone duality theory of bounded distributive lattices that $\{D_{Id}(\lambda_A(\theta))|\theta\in C(A)\}$ is the basis of compact open sets for the spectral space $Spec_{Id,Z}(L(A))$ and this basis is closed under finite intersections. Thus $\mathcal{B}= \{ u^{-1}(D_{Id}(\lambda_A(\theta)))|\theta\in C(A)\}$ is the basis of compact open sets for the spectral space $Spec_Z(A)$ and it is closed under finite intersections.

Let us consider a congruence $\theta\in C(A)$. By applying Lemmas 3.11(1) and 3.23 for the lattice ideal $I=(\lambda_A(\theta)]$, we get $u^{-1}(D_{Id}(\lambda_A(\theta)))= D((\lambda_A(\theta)]_{\ast})= D(\theta)$, therefore $D(\theta)\in \mathcal{B}$. We conclude that $A$ is a spectral algebra.
\end{proof}

According to Proposition 3.14, if $A\in \mathcal{V}$ is quasi-commutative then $Spec_Z(A)$ and $Spec_{Id,Z}(L(A))$ are homeomorphic spaces.
Then the quasi-commutative algebras constitute a suitable class of algebras for defining a notion of reticulation with remarkable transfer properties. In conclusion, for any quasi-commutative algebra $A\in \mathcal{V}$, the bounded distributive lattice $L(A)$ will be called the reticulation of $A$.

\begin{remarca} Assume that $A\in \mathcal{V}$ is an algebra with the property that $K(A)$ is closed under the commutator operation. Then $C(A)= K(A)$ so $A$ is a quasi-commutative algebra: if $\alpha,\beta\in K(A)$ then, by taking $\gamma= [\alpha,\beta]\in K(A)$, it follows that $\gamma\subseteq [\alpha,\beta]$
and $\rho(\gamma)=\rho([\alpha,\beta])$.
\end{remarca}

Following \cite{Kaplansky}, p. 73, a ring $R$ is said to be neo-commutative if the product of two finitely generated ideals is finitely generated. Then the algebras $A\in \mathcal{V}$ with the property that $K(A)$ is closed under commutator operation generalize the neo-commutative rings. Kaplansky proved in  \cite{Kaplansky} that any neo-commutative ring is spectral. Thus Remark 3.27 generalizes the Kaplansky result to the algebras of variety $\mathcal{V}$.

\begin{propozitie} The following algebras of the variety $\mathcal{V}$ are quasi-commutative:
\usecounter{nr}
\begin{list}{(\arabic{nr})}{\usecounter{nr}}
\item algebras in which any congruence is compact;
\item algebras with finitely many prime congruences.
\end{list}
\end{propozitie}

\begin{proof} $(1)$ If $A$ is an algebra of $\mathcal{V}$ such that $Con(A)=K(A)$ then $K(A)=C(A)= Con(A)$. By Remark 3.27, $A$ is quasi-commutative.

$(2)$  Let $A$ be an algebra of $\mathcal{V}$ such that $Spec(A)$ is finite. Then $D(\theta)$ is compact, for each $\theta\in Con(A)$, so $A$ is a spectral algebra. By applying Theorem 3.26, it follows that $A$ is quasi-commutative.
\end{proof}

\section{Boolean elements and annihilators versus reticulation}

\hspace{0.5cm} Let $\mathcal{V}$ be a semidegenerate congruence-modular variety and $A$ an algebra of $\mathcal{V}$.

Let $B(Con(A))$ be the Boolean algebra of complemented congruences in the algebra $A$ (cf. Lemma 4 from \cite{Jipsen}). For each $\alpha\in B(Con(A))$, the annihilator $\alpha^{\perp}$ is the complement of $\alpha$. A congruence $\alpha$ of $A$ is complemented if and only if $\alpha\lor \alpha^{\perp} = \nabla_A$. If $\theta,\vartheta\in Con(A)$ and $\alpha\in B(Con(A))$ then $\theta\land \alpha =[\theta,\alpha]$, $\alpha\rightarrow \theta  = \alpha^{\perp}\lor \theta$ and $(\theta\land\vartheta)\lor \alpha = (\theta\lor \alpha)\land (\vartheta\lor \alpha)$.

\begin{lema}\cite{GM2}
 For all congruences $\theta, \vartheta\in Con(A)$ the following hold:
\usecounter{nr}
\begin{list}{(\arabic{nr})}{\usecounter{nr}}
\item If $\theta\lor \vartheta = \nabla_A$ and $[\theta,\vartheta] = \Delta_A$ then $\theta,\vartheta\in B(Con(A))$;
\item $B(Con(A))\subseteq K(A)$.

\end{list}
\end{lema}

According to the previous lemma, $B(Con(A))\subseteq C(A)$. Recall from \cite{Birkhoff} that the Boolean center of a bounded distributive lattice $L$ is the Boolean algebra $B(L)$ of the complemented elements in $L$. Similarly, we shall say that $B(Con(A))$ is the Boolean center of $Con(A)$.

Firstly, we shall establish a relationship between the Boolean algebra $B(Con(A))$ and $B(L(A))$.

\begin{lema} If $\alpha\in B(Con(A))$ then $\lambda_A(\alpha)\in B(L(A))$.
\end{lema}

\begin{proof} If $\alpha\in B(Con(A))$ then $\alpha\lor \beta= \nabla_A$ and  $\alpha\cap \beta= \Delta_A$, for some $\beta\in B(Con(A))$. By Lemma 3.2(1) and (2) we get $\lambda_A(\alpha)\lor \lambda_A(\beta)= 1$ and $\lambda_A(\alpha)\land \lambda_A(\beta)= 0$, so $\lambda_A(\alpha)\in B(L(A))$.
\end{proof}

Lemma 4.2 allows us to consider the function

$\lambda_A|_{B(Con(A))}:B(Con(A))\rightarrow B(L(A))$.

\begin{lema} $\lambda_A|_{B(Con(A))}: B(Con(A))\rightarrow B(L(A))$ is an injective Boolean morphism.
\end{lema}

\begin{proof} The properties (1) and (2) of Lemma 3.2 show that $\lambda_A|_{B(Con(A))}$ is a Boolean morphism. The injectivity follows by using Lemma 3.2(3).
\end{proof}

\begin{propozitie}
Assume that the algebra $A$ is semiprime. Then the function $\lambda_A|_{B(Con(A))}:B(Con(A))\rightarrow B(L(A))$ is a Boolean isomorphism.
\end{propozitie}

\begin{proof}
According to Lemma 4.3, it suffices to prove that the function $\lambda_A|_{B(Con(A)}$ is surjective. Assume that $x\in B(L(A))$ so there exists $y\in B(L(A))$ such that $x\lor y=1$ and $x\land y = 0$. Thus there exist $\alpha,\beta\in C(A)$ such that $x=\lambda_A(\alpha), y= \lambda_A(\beta)$, so $\lambda_A(\alpha\lor \beta)= \lambda_A(\alpha)\lor \lambda_A(\beta) = 1$ and  $\lambda_A([\alpha, \beta])= \lambda_A(\alpha)\land  \lambda_A(\beta) = 0$. By Lemma 3.2(3) we obtain $\alpha\lor \beta= \nabla_A$. Since $A$ is semiprime, $\lambda_A([\alpha, \beta])= 0$ implies $[\alpha,\beta] = \Delta_A$ (cf. Lemma 3.2(5)). Due to Lemma 4.1(1) we get $\alpha\in B(Con(A))$, therefore $\lambda_A|_{B(Con(A))}$ is surjective.

\end{proof}

Recall from \cite{GM2} that the algebra $A$ is said to be hyperarchimedean if for any $\alpha\in K(A)$ there exists an integer $n\geq 1$ such that $[\alpha,\alpha]^n\in B(Con(A))$.

\begin{propozitie} Any hyperarchimedean semiprime algebra $A$ is quasi-commutative.
\end{propozitie}

\begin{proof} Assume that $\alpha,\beta \in K(A)$ so there exists an integer $n\geq 1$ such that $[\alpha,\alpha]^n\in B(Con(A))$ and $[\beta,\beta]^n\in B(Con(A))$. Denoting $\gamma=[[\alpha,\alpha]^n, [\beta,\beta]^n]$ it follows that $\gamma\in B(Con(A))\subseteq K(A)$, $\gamma\subseteq [\alpha,\beta]$ and $\lambda_A(\gamma)= \lambda_A([\alpha,\alpha]^n)\land \lambda_A( [\beta,\beta]^n])= \lambda_A(\alpha)\land \lambda_A(\beta)= \lambda_A([\alpha,\beta])$. In virtue of Lemma 3.20, $A$ is quasi-commutative.
\end{proof}

If $L$ is a bounded distributive lattice and $I$ is an ideal of $L$ then the annihilator of $I$ is the ideal $Ann(I) = \{x\in L|x\land y = 0$ for all $y\in I\}$.

\begin{lema}
If $\alpha\in C(A)$ and $\phi\in Spec(A)$ then $\lambda_A(\alpha)\in \phi^{\ast}$ if and only if $\alpha\subseteq \phi$.
\end{lema}

\begin{proof} Similar to the proof of Lemma 4.42 of \cite{GKM}.
\end{proof}

The following two results improve Propositions 6.25 and 6.26 of \cite{GM2}. They show the way in which the functions $(\cdot)^{\ast}$ and $(\cdot)_{\ast}$ preserve the annihilators.

\begin{propozitie} Assume that the algebra $A$ is semiprime.
If $\theta\in Con(A)$ then $Ann(\theta^{\ast}) = (\theta^{\perp})^{\ast}$.
\end{propozitie}

\begin{proof} Firstly we shall prove that $Ann(\theta^{\ast})\subseteq (\theta^{\perp})^{\ast}$. Let $x$ be an element of $Ann(\theta^{\ast})$, so $x = \lambda_A(\alpha)$ for some $\alpha\in C(A)$. Assume that $\beta\in K(A)$ and $\beta\subseteq \theta$, therefore $\lambda_A([\alpha,\beta]) = \lambda_A(\alpha)\land \lambda_A(\beta) = 0$. But $A$ is semiprime, so $[\alpha,\beta] =\Delta_A$ (by Lemma 3.2(5)), hence $\alpha\subseteq \beta^{\perp}$. Then the following hold:

$\alpha\subseteq \bigcap\{\beta^{\perp}|\beta\in K(A), \beta\subseteq \theta\}$ = $(\bigvee\{\beta\in K(A)|\beta\subseteq \theta\})^{\perp}$ = $ \theta^{\perp}$.

It follows that $x = \lambda_A(\alpha)\in (\theta^{\perp})^{\ast}$, hence the inclusion $Ann(\theta^{\ast})\subseteq (\theta^{\perp})^{\ast}$ is proven.

For establish the converse inclusion $(\theta^{\perp})^{\ast}\subseteq Ann(\theta^{\ast})$, consider an element $x\in (\theta^{\perp})^{\ast}$, hence $x = \lambda_A(\alpha)$ for some $\alpha\in C(A)$ such that $\alpha\subseteq \theta^{\perp}$. Let $y$ be an element of $ \theta^{\ast}$. We have to prove that $x\land y=0$. $y$ has the form $y=\lambda_A(\beta)$, for some $\beta\in C(A)$ such that $\beta\subseteq \theta$. Since $\alpha\subseteq \theta^{\perp}$ we obtain $[\alpha,\beta]=0$, hence by applying Lemma 3.2(2) we get $x\land y= \lambda_A(\alpha)\land \lambda_A(\beta)=\lambda_A([\alpha,\beta])=0$. It results that $x\in Ann(\theta^{\ast})$, so the converse inclusion $(\theta^{\perp})^{\ast}\subseteq Ann(\theta^{\ast})$ holds.
\end{proof}

\begin{propozitie}
Assume that the algebra $A$ is semiprime. If $I$ is an ideal of the lattice $L(A)$ then $(Ann(I))_{\ast} = (I_{\ast})^{\perp}$.
\end{propozitie}

\begin{proof} Let $\alpha$ be a compact congruence of $A$. By using the distributivity of commutator operation we have

$[\alpha,I_{\ast}]=\bigvee\{[\alpha,\beta]|\beta\in K(A),\lambda_A(\beta)\in I\}$.

According to Lemma 3.2(5), the following assertions are equivalent:

$\bullet$ $\alpha\subseteq (I_{\ast})^{\perp}$;

$\bullet$ $[\alpha,I_{\ast}]=\Delta_A$;

$\bullet$ For all $\beta\in K(A)$, $\lambda_A(\beta)\in I$ implies $[\alpha,\beta]=\Delta_A$;

$\bullet$ For all $\beta\in K(A)$, $\lambda_A(\beta)\in I$ implies $\lambda_A([\alpha,\beta])=0$;

$\bullet$ For all $\beta\in K(A)$, $\lambda_A(\beta)\in I$ implies $\lambda_A(\alpha)\land \lambda_A(\beta)=\Delta_A$;

$\bullet$ $\lambda_A(\alpha)\in Ann(I)$;

$\bullet$ $\alpha\subseteq (Ann(I))_{\ast}$.

In virtue of the previous equivalences,  $\alpha\in (I_{\ast})^{\perp}$ if and only if $\alpha\subseteq (Ann(I))_{\ast}$, therefore $(Ann(I))_{\ast} = (I_{\ast})^{\perp}$.
\end{proof}

Let $Min(A)$ be the set of minimal prime congruences of the algebra $A$ ($Min(A)$ is named the minimal prime spectrum of $A$). If $\phi\in Spec(A)$ then there exists $\psi\in Min(A)$ such that $\psi\subseteq \phi$. We shall denote by $Min_Z(A)$ the topological space obtained by restricting the topology of $Spec_Z(A)$ to $Min(A)$.

If $L$ is a bounded distributive lattice then $Min_{Id}(L)$ will denote the set of minimal prime ideals of $L$ (= the minimal prime spectrum of $L$). $Min_{Id,Z}(L)$ will be the topological space obtained by restricting the topology of $Spec_{Id,Z}(L)$ to $Min_{Id}(L)$.

\begin{lema}
Let $A$ be a quasi-commutative algebra. Assume that $\phi\in Spec(A)$ and $P\in Spec_{Id}(L(A))$.
\usecounter{nr}
\begin{list}{(\arabic{nr})}{\usecounter{nr}}
\item $\phi\in Min(A)$ if and only if $\phi^{\ast}\in Min_{Id}(L(A))$;
\item $P\in Min_{Id}(L(A))$ if and only if $P_{\ast}\in Min(A)$.

\end{list}
\end{lema}

\begin{proof} According to Theorem 3.26, for any ideal $I$ of $L(A)$, we have $I=(I_{\ast})^{\ast}$, so the conditions of Proposition 3.12 are satisfied. Equivalences (1) and (2) follow by using that the order-preserving isomorphisms $u$ and $v$ from Proposition 3.16 are inverse to one another.
\end{proof}

By the previous lemma it follows that the bijective maps

$u|_{Min(A)}:Min(A)\rightarrow Min_{Id}(L(A))$, $v|_{Min_{Id}(L(A))}:Min_{Id}(L(A))\rightarrow Min(A)$ are inverse to one another.

\begin{lema} If $A$ is a quasi-commutative algebra then the following maps

$u|_{Min(A)}:Min(A)\rightarrow Min_{Id}(L(A))$; $v|_{Min_{Id}(L(A))}:Min_{Id}(L(A))\rightarrow Min(A)$ are homeomorphisms, inverse to one another.
\end{lema}

\begin{proof}
From the previous lemma we already know that the maps $u|_{Min(A)}$ and $v|_{Min_{Id}(L(A))}$ are bijections inverse one to another. To prove that they are continuous is straightforward, by using that the maps $u$ and $v$ from Proposition 3.16 are homeomorphisms, inverse to one another.
\end{proof}

\begin{corolar}
If $A$ is a quasi-commutative algebra then $Min_Z(A)$ is a zero-dimensional Hausdorff space.
\end{corolar}

\begin{proof}
 We know from \cite{Speed} that the minimal prime spectrum of a bounded distributive lattice is a zero-dimensional Hausdorff space. In particular, the minimal prime spectrum $Min_{Id,Z}(L(A))$ is a zero-dimensional Hausdorff space. Since $Min_Z(A)$ and $Min_{Id,Z}(L(A))$ are homeomorphic spaces (cf. Lemma 4.10) it follows that $Min_Z(A)$ is a zero-dimensional Hausdorff space.
\end{proof}

\begin{lema}\cite{Speed}
Let $P$ be a prime ideal of a bounded distributive lattice $L$. Then the following are equivalent:
\usecounter{nr}
\begin{list}{(\arabic{nr})}{\usecounter{nr}}
\item $P$ is a minimal prime ideal;
\item For any $x\in L$, $x\in P$ implies $Ann(x)\not\subseteq P$;
\item For any $x\in L$, $x\in P$ if and only if $Ann(x)\not\subseteq P$.

\end{list}
\end{lema}

Now we shall use Lemma 4.12 and the transfer properties contained in Propositions 4.7 and 4.8 in order to characterize the minimal prime congruences in a semiprime quasi-commutative algebra.

\begin{lema} If $\theta\in Con(A)$ and $\phi\in Spec(A)$ then $\theta\subseteq \phi$ if and only if $\theta^{\ast}\subseteq \phi^{\ast}$
\end{lema}

\begin{proof} Assume that $\theta^{\ast}\subseteq \phi^{\ast}$. By using Lemmas 3.6 and 3.7 we have $\theta\subseteq (\theta^{\ast})_{\ast}\subseteq (\phi^{\ast})_{\ast}=\phi$. The converse implication is obvious.
\end{proof}

\begin{teorema}
Let $\phi$ be a prime congruence of a semiprime quasi-commutative algebra $A$. Then the following are equivalent:
\usecounter{nr}
\begin{list}{(\arabic{nr})}{\usecounter{nr}}
\item $\phi$ is a minimal prime congruence;
\item For any $\alpha\in C(A)$, $\alpha\subseteq \phi$ implies $\alpha^{\perp}\not\subseteq \phi$;
\item For any $\alpha \in C(A)$, $\alpha\subseteq \phi$ if and only if $\alpha^{\perp}\not\subseteq \phi$.
\item For any $\alpha\in K(A)$, $\alpha\subseteq \phi$ implies $\alpha^{\perp}\not\subseteq \phi$;
\item For any $\alpha \in K(A)$, $\alpha\subseteq \phi$ if and only if $\alpha^{\perp}\not\subseteq \phi$.

\end{list}
\end{teorema}

\begin{proof}
$(1) \Leftrightarrow (2)$ According to Lemma 3.4, for any $\alpha\in C(A)$, $\alpha^{\ast}$ is the principal ideal $(\lambda_A(\alpha)]$ generated by $\{\lambda_A(\alpha\}$ in the lattice $L(A)$. Since $A$ is semiprime, for any $\alpha\in C(A)$ we have $Ann(\lambda_A(\alpha))= Ann((\lambda_A(\alpha)])= Ann(\alpha^{\ast})=(\alpha^{\perp})^{\ast}$ (cf. Proposition 4.7).

By taking into account that $A$ is a semiprime quasi-commutative algebra and by using Lemmas 4.9, 4.12 and 4.13 we get the equivalence of the following properties:

$\bullet$ $\phi$ is a minimal prime congruence of $A$;

$\bullet$ $\phi^{\ast}$ is a minimal prime ideal of $L(A)$;

$\bullet$ For all $x\in\phi^{\ast}$, $Ann(x)\not\subseteq \phi^{\ast}$;

$\bullet$ For all $\alpha\in C(A)$, $\alpha\subseteq \phi$ implies $Ann(\lambda_A(\alpha))\not\subseteq \phi^{\ast}$;

$\bullet$ For all $\alpha\in C(A)$, $\alpha\subseteq \phi$ implies $(\alpha^{\perp})^{\ast}\not\subseteq \phi^{\ast}$;

$\bullet$ For all $\alpha\in C(A)$, $\alpha\subseteq \phi$ implies $\alpha^{\perp}\not\subseteq \phi$.

$(2) \Rightarrow (4)$ Obviously.

$(4) \Rightarrow (2)$ We shall prove the property $(4)$ by induction on the way in which the set $C(A)$ is defined :

$\bullet$ If $\alpha\in K(A)$ then we apply the hypothesis $(2)$.

$\bullet$ Assume that $\alpha=\beta\lor\gamma$, where $\beta,\gamma\in C(A)$ satisfy the induction hypothesis. Then $\alpha\subseteq \phi$ imply $\beta\subseteq \phi$ and $\gamma\subseteq \phi$, so $\beta^{\perp}\not\subseteq \phi$ and $\gamma^{\perp}\not\subseteq \phi$. Since $\phi$ is a prime congruence, it follows that $[\beta^{\perp},\gamma^{\perp}\not\subseteq\phi$, hence $\alpha^{\perp}= (\beta\lor\gamma)^{\perp}= \beta^{\perp}\cap \gamma^{\perp}\not\subseteq \phi$.

$\bullet$ Assume that $\alpha=[\beta,\gamma]$, where the congruences $\beta,\gamma\in C(A)$ satisfy the induction hypothesis. Then $\alpha\subseteq \phi$ implies $\beta\subseteq \phi$ or $\gamma\subseteq \phi$ (because $\phi\in Spec(A))$. If $\beta\subseteq \phi$ then $\beta^{\perp}\not\subseteq \phi$. From $\alpha\subseteq \beta$ we get $\beta^{\perp}\subseteq \alpha^{\perp}$, hence $\alpha^{\perp}\not\subseteq \phi$. The case $\gamma\subseteq \phi$ is treated in a similar way.

The equivalences $(2) \Leftrightarrow (3)$ and $(4) \Leftrightarrow (5)$ follow immediately.

\end{proof}

\section{Concluding Remarks}

\hspace{0.5cm} In this paper we introduced two classes of algebras in a semidegenerate congruence-modular variety $\mathcal{V}$: the quasi-commutative algebras and the spectral algebras. The first one generalizes the quasi-commutative rings and the second one generalizes the spectral rings (\cite{Belluce}, \cite{Belluce1}). We enumerate the main contributions:

$\bullet$ the quasi-commutative algebras and the spectral algebras coincide;

$\bullet$ the quasi-commutative algebras give us a suitable framework to define and study a notion of reticulation;

$\bullet$ we emphasize that this reticulation allows us to transfer some properties from bounded distributive lattices to quasi-commutative algebras and vice-versa (see the results proven in Section 4).

Now we shall mention some open problems:

$(1)$ It is well - known that the reticulation of commutative rings ensures a covariant functor from the category of commutative rings to the category of bounded distributive lattices (see e.g. \cite{Johnstone}). If we consider the variety $\mathcal{V}$ as a category (whose morphisms are $\tau$ - morphisms) then a similar result for reticulation construction does not hold (see \cite{GKM}). In \cite{GKM} it was studied a reticulation functor defined on the category $\mathcal{V^{tilde}}$ whose objects are the algebras of $\mathcal{V}$ and morphisms are the admissible morphisms of $\mathcal{V}$.

 Let us denote by $\mathcal{V}^{\star}$ the following category: (1) the objects of $\mathcal{V}^{\star}$ are the quasi-commutative algebras of $\mathcal{V}$; (2) the morphisms of $\mathcal{V}^{\star}$ are the admissible morphisms between the quasi-commutative algebras of $\mathcal{V}$.

 Now we shall state without proof a proposition that shows how a morphism $u:A\rightarrow B$ of $\mathcal{V}^{\star}$ provides a morphism $L(u):L(A)\rightarrow L(B)$ in the category of bounded distributive lattices. This  proposition generalizes a result of \cite{GM2} and their proofs are very similar.

\begin{propozitie}
Assume that $u:A\rightarrow B$ is a morphism in $\mathcal{V}^{\star}$ . Then there exists a morphism of bounded distributive lattices $L(u): L(A)\rightarrow L(B)$ such that the following diagram is commutative:

\begin{center}
\begin{picture}(150,70)
\put(0,50){$C(A)$}
\put(25,55){\vector(1,0){100}}
\put(70,60){$u^{\bullet}|C(A)$}
\put(130,50){$C(B)$}
\put(5,45){\vector(0,-1){30}}
\put(-10,30){$\lambda_{A}$}
\put(0,0){$L(A)$}
\put(25,5){\vector(1,0){100}}
\put(70,10){$L(u)$}
\put(130,0){$L(B)$}
\put(135,45){\vector(0,-1){30}}
\put(140,30){$\lambda_{B}$}
\end{picture}
\end{center}

\end{propozitie}

By using the previous proposition it is easy to see that the assignments $A\mapsto L(A)$ and $u\mapsto L(u)$ define a covariant functor $L$ defined from the category $\mathcal{V}^{\star}$ to the category $D_{0,1}$ of bounded distributive lattices. The functor $L:\mathcal{V}^{\star}\rightarrow D_{0,1}$ will be called the reticulation functor. The paper \cite{GKM} contains a lot of properties and transfer properties of reticulation functor $L:\mathcal{V^{tilde}}\rightarrow D_{0,1}$ defined in \cite{GKM}. In a future paper we shall investigate how the results of \cite{GM1} and \cite{GKM} can be extended to the reticulation functor $L:\mathcal{V}^{\star}\rightarrow D_{0,1}$.

$(2)$ Recall Proposition 40 from \cite{Belluce1}: a semiprime ring $R$ is quasi-commutative if and only if the prime spectrum $Spec(R)$ is a spectral space. In \cite{Klep} it was proven that it is not true. An interesting question is to characterize the subclass of $\mathcal{V}$ whose members verify this equivalence.

$(3)$ By Proposition 4.4, for any semiprime algebra $A\in \mathcal{V}$, the Boolean algebras $B(Con(A))$ and $B(L(A))$ are isomorphic. We have no example of algebra $A\in \mathcal{V}$ or $A\in \mathcal{V}^{\star}$ for which $B(Con(A))$ and $B(L(A))$ are not isomorphic. Characterizes the class of algebras in $\mathcal{V}$ for which Proposition 4.4 remains valid. Is this class larger than the class of semiprime algebras?

$(4)$ According to a famous Hochster theorem \cite{Hochster}, for each bounded distributive lattice $L$ there exists a commutative ring $R$ such that the reticulation $L(R)$ of $R$ and $L$ are isomorphic lattices. Then for each quasi-commutative algebra $A\in \mathcal{V}$ there exists a commutative ring $R$ such that the reticulations $L(A)$ and $L(R)$ are isomorphic. In this way we obtain a bridge between the quasi-commutative algebras of $\mathcal{V}$ and the commutative rings.

Can we use this bridge for exporting results from commutative rings to quasi-commutative algebras? For example, we think that the theory of Gelfand commutative rings (\cite{Agliano}, \cite{Johnstone}) can be transported in a theory of congruence-normal quasi-commutative algebras (see Section 8 of \cite{GM1}).

$(5)$  Can we use the reticulation for obtaining sheaf representations for the quasi-commutative algebras of $\mathcal{V}$? It would be interesting if using Hochster theorem and reticulation, we could transfer  sheaf representations of commutative rings (see \cite{Johnstone}) to sheaf representations of  quasi-commutative algebras of $\mathcal{V}$.

$(6)$ Given a semidegenerate congruence-modular variety $\mathcal{V}$ we can formulate the following "Spectrum Problem": characterize the bounded distributive lattices $L$ for which there exists a quasi-commutative algebra $A\in \mathcal{V}$ such that the reticulation $L(A)$ of $A$ is isomorphic with $L$. In the case of commutative rings, the Spectrum Problem is solved by the Hochster theorem \cite{Hochster}. Recently, Lenzi and Di Nola gave a solution of the Spectrum Problem for MV-algebras (see \cite{Lenzi}). To solve Spectrum Problem corresponding to a semidegenerate congruence-modular variety $\mathcal{V}$ is an important open question.


\begin{thebibliography}{200}
\bibitem{Aghajani} M. Aghajani, A. Tarizadeh, Characterization of Gelfand rings, specially clean rings and their dual rings, Results Math.,75:125,2020
\bibitem{Agliano} P. Agliano, Prime spectra in modular varieties, Algebra Universalis, 30, 1993, 581 - 597
\bibitem{Atiyah} M. F. Atiyah, I. G. MacDonald, Introduction to Commutative Algebra, Addison-Wesley Publ. Comp., 1969

\bibitem{Al-Ezeh2} H. Al- Ezeh, Further results on reticulated rings, Acta Math. Hung., 60 (1-2), 1992, 1 - 6
\bibitem{BalbesDwinger} R. Balbes, Ph. Dwinger, Distributive Lattices, Univ. of Missouri Press, 1974
\bibitem{Belluce0} L. P. Belluce, Semisimple algebras of infinite valued logic and bold fuzzy set theory, Canadian J. Math., 38, 1986, 1356 - 1379
\bibitem{Belluce} L. P. Belluce, Spectral spaces and non-commutative rings, Communications in Algebra, 19(7), 1991, 1855 - 1865
\bibitem{Belluce1} L. P. Belluce, Spectral closure for non-commutative rings, Communications in Algebra, 25(5), 1997, !513 - 1536
\bibitem{Birkhoff} G. Birkhoff, Lattice Theory, 3rd ed., AMS Collocquium Publ. Vol. 25, 1967
\bibitem{Burris} S. Burris, H. P. Sankappanavar, A Course in Universal Algebra, Graduate Texts in Mathematics, 78, Springer Verlag, 1881

\bibitem{Dickmann} M. Dickmann, N. Schwartz, M. Tressl, Spectral Spaces, Cambridge Univ. Press., 2019

\bibitem{Fresee} R. Fresee, R. McKenzie, Commutator Theory for Congruence Modular Varieties, Cambridge Univ. Press, 1987
\bibitem{Galatos} N. Galatos, P. Jipsen, T. Kowalski, H. Ono, Residuated Lattices: An Algebraic Glimpse at Structural Logics, Studies in Logic and The Foundation of Mathematics, 151, Elsevier, 2007
\bibitem{GM1} G. Georgescu, Reticulation functor and the transfer properties, ArXiv: 2205.02174v1[math.LO] 4 May 2022
\bibitem{GM2} G. Georgescu, C. Mure\c{s}an, The reticulation of a universal algebra, Scientific Annals of Computer Science, 28, 2018, 67 - 113

\bibitem{GKM} G. Georgescu, L. Kwuida, C. Mure\c{s}an, Functorial properties of the reticulation of a universal algebra, J. Applied Logic, 8(5), 2021, 102 - 132

\bibitem{Hochster} M. Hochster, Prime ideals structures in commutative rings, Trans.Amer.Math.Soc., 142, 1969, 43 - 60
\bibitem{Jipsen} P. Jipsen, Generalization of Boolean products for lattice-ordered algebras, Annals Pure Appl. Logic, 161, 2009, 224 - 234
\bibitem{Johnstone} P. T. Johnstone, Stone Spaces, Cambridge Univ. Press, 1982
\bibitem{Kaplansky} I. Kaplansky, Topics in Commutative Ring Theory, Dept. of Math, University of Chicago
\bibitem{Klep} J. Klep, M. Tressl, The prime spectrum and the extended prime spectrum of noncommutative rings, Algebra Represent. Theory, 2017, 10:257 - 270
\bibitem{Kollar} J. Kollar, Congruences and one - element subalgebras, Algebra Universalis, 9, 1979, 266 - 276
\bibitem{Kowalski} T. Kowalski, H. Ono, Residuated Lattices: An Algebraic Glimpse at Logics without Contraction, manuscript, 2000
\bibitem{Lenzi} G. Lenzi, A. Di Nola, The spectrum problem for abelian $l$ - groups and $MV$ - algebras, Algebra Universalis, 81(3), 2020
\bibitem{Leustean} L. Leu\c{s}tean, The maximal and prime spectra of BL-algebras and the reticulation of BL-algebras, Central European J. Math.,1(3), 2003, 382 - 397
\bibitem{g} L. Leu\c{s}tean, Representations of many-valued algebras, Editura Universitara, Bucharest, 2010
\bibitem{Muresan} C. Mure\c{s}an, Algebras of many-valued logic. Contributions to the theory of residuated lattices, Ph.D.Thesis, Bucharest University, 2009
\bibitem{Simmons} H. Simmons, Reticulated rings, J. Algebra, 66, 1980, 169 - 192
\bibitem{Speed} T. P. Speed, Spaces of ideals of distributive lattices II. Minimal prime ideals, J. Australian Math. Soc., 18, 1974 - 72

\end{thebibliography}
\end{document}